\documentclass[11pt,reqno]{amsart}

\usepackage{amsmath,amssymb,amsfonts,mathrsfs,verbatim,enumitem,pstricks,amsthm, graphicx,setspace}
\usepackage[margin=2.5cm]{geometry}
\usepackage[all]{xy}
\usepackage{centernot, xr, accents}
\usepackage{hyperref,comment}

\usepackage{amsmath,amssymb,mathtools}
\usepackage[all]{xy}
\usepackage{hyperref}
\usepackage{verbatim}
\allowdisplaybreaks
\usepackage[bottom]{footmisc}
\usepackage{lipsum,cleveref}

\usepackage{hyperref}

\theoremstyle{plain}

\newtheorem{thm}{Theorem}[section]
\newtheorem{cor}[thm]{Corollary}
\newtheorem{lmm}[thm]{Lemma}
\newtheorem{prp}[thm]{Proposition}

\newtheorem{que}[thm]{Question}

\newtheorem*{mthm}{Main Theorem}
\newtheorem*{app}{Application}

\theoremstyle{definition}
\newtheorem{defn}[thm]{Definition}

\newtheorem{rem}{Remark}

\newtheorem{eg}[thm]{Example}

\theoremstyle{definition}

\numberwithin{equation}{section}

\newcommand{\C}{\mathbb{C}}

\newcommand{\cO}{\mathcal{O}}
\newcommand{\cH}{\mathcal{H}}

\def \hf{\hspace*{0.5cm}}

\begin{document}
\title[An interpolation problem and its applications]{A regular interpolation problem and its applications}
\author[N. Das]{Nilkantha Das}

\address{Stat-Math Unit, Indian Statistical Institute, 203 B.T. Road, Kolkata 700 108, India.}
\email{dasnilkantha17@gmail.com}

\begin{abstract}
In this article, we prove the following interpolation problem: if the composition of a function and a regular map between affine varieties is a regular function, then there exists a global regular function of the target variety that coincide with the function on the image of the regular map provided the target variety is factorial and the regular map is almost surjective. We also discuss a few applications of the interpolation problem.
\end{abstract}

\maketitle
\tableofcontents

\section{Introduction}\label{intro}
Let $k$ be an algebraically closed field of characteristic zero, and $\Phi:X \longrightarrow Y$ be a regular map between affine varieties $X$ and $Y$. Consider the following problem:
\begin{que}
If $f \circ \Phi$ is a regular function on $X$ for some function $f$ on $Y$, is $f$ necessarily a regular function on $Y ?$
\end{que}
\noindent The answer to the above problem is negative in general (\Cref{counter-example}) as we don't have control outside $\text{Im}(\Phi)$, the range of $\Phi$. It is natural to ask the following:
\begin{que}
If $f \circ \Phi$ is a regular function on $X$ for some function $f$ on $Y$,  does there exist a regular function $g$ on $Y$ such that $f\mid_{\text{Im}(\Phi)}=g \mid_{\text{Im}(\Phi)}?$ 
\end{que}
In other words, the above question asks if $f$ can be interpolated over $\text{Im}(\Phi)$ by a regular function on $Y$. Throughout this article, we call this problem as interpolation problem. The interpolation problem does not always have an affirmative answer in general (\Cref{counter-example}) as the size of $\text{Im}(\Phi)$ with respect to $Y$ may be too small. To make this precise, we extend the definition of being \textit{almost surjective} from \cite[Definition 2.1]{Aichinger} to arbitrary varieties.
\begin{defn}
Let $X,Y$ be two affine varieties over $k$. A regular map $\Phi \colon X \longrightarrow Y$ is said to be \textit{almost surjective} if the dimension of the Zariski closure of the complement of image of $\Phi$ is at most $\dim Y -2$, that is, $\dim \overline{\left(Y \setminus \text{Im}(\Phi ) \right)} \leq \dim Y -2$. 
\end{defn} 
\noindent The main result of the paper states that if $Y$ is factorial and $\Phi$ is almost surjective, the interpolation problem has an affirmative answer. 
\begin{mthm}[\Cref{fun_comp_corollary}]
Let $X,Y$ be two affine varieties over an algebraically closed field $k$ of characteristic zero such that $Y$ is factorial, and $\Phi : X \longrightarrow Y$ be an almost surjective morphism. Furthermore, assume $h : Y \rightarrow k $ is a function such that $h \circ \Phi$ is a regular function on $X$. Then there exists a regular function $p:Y \rightarrow k$ such that $p(y)=h(y)$ for all $y \in \Phi (X)$.
\end{mthm}
Our result is a generalization of a result of Aichinger (\cite[Theorem 4.1]{Aichinger}) who proved it for affine spaces. It is worthwhile to note that if $X=Y=\mathbb{A}^1_{\mathbb{C}}$, this interpolation problem can be proved using one variable complex calculus (cf. \cite[Exercise 14, Chapter 10]{Rudin}).
\begin{eg}\label{counter-example}
Let $\Phi: \mathbb{A}^1_{k} \longrightarrow \mathbb{A}^2_{k}$ be given by $\Phi (t)=(t^2,t^3)$, and $f:\mathbb{A}^2_{k} \longrightarrow k$ be given by
\begin{align*}
f(x,y)= \begin{cases}
y/x & \text{ if } x \neq 0,\\
0 & \text{ otherwise.}
\end{cases}
\end{align*}
Here $f \circ \Phi$ is a regular function, but $f$ can't be interpolated by a polynomial in two variables.
\end{eg}
\noindent \textbf{Notation:}  Throughout the article, we always assume $X\subseteq \mathbb{A}^m_k$ and $Y \subseteq \mathbb{A}^n_k$ are two affine varieties with the algebraic structure sheaves $\cO_X$ and $\cO_Y$, respectively. By affine variety, we mean an irreducible algebraic set. The field $k$ is always assumed to be an algebraically closed field of characteristic zero, unless otherwise specified. The coordinate rings of $X$ and $Y$ are denoted by $A(X)$ and $A(Y)$, respectively. The morphism $\Phi:X \rightarrow Y$ is defined by $\phi_1, \ldots ,\phi_n$, where $\phi_i \in A(X)$ for $1 \leq i \leq n$. We occasionally deal with the complex varieties. For a complex algebraic variety $(X,\cO_X)$, the associated analytic space is denoted by $(X^{\text{an}},\cH_X) $. The open sets in the Zariski topology and analytic topology will be denoted by Zariski open and open, respectively, in order to distinguish the topologies as well. For a sheaf $\mathcal{F}$ on a topological space $X$, the ring associated to an open subset $U$ of $X$ is denoted by $\Gamma\left( U,\mathcal{F}\right)$. \\
\hf Let us analyze the interpolation problem in a different way. We say a regular function $g$ on $X$ is \textit{determined by $\Phi$} (cf. \cite{Aichinger}), if there is a function $h : Y \rightarrow k$ (not necessarily regular) such that $h \circ \Phi=g$. The set of all such regular functions that are determined by $\Phi$ is a $k$-subalgebra of $A(X)$; it is denoted by $A\langle \Phi \rangle$. On the other hand, let us denote the image of $A(Y)$ via the induced $k$-algebra morphism $\Phi^*: A(Y) \longrightarrow A(X)$ by $A[\Phi]$. Observe that affirmative answer to the interpolation problem is the same amount of saying that $A[\Phi]= A\langle \Phi \rangle$. Note that $A[\Phi] \subseteq A\langle \Phi \rangle$, but the reverse inclusion does not hold in general. We show that if  $Y$ is factorial and $\Phi$ is almost surjective,  then $A[ \Phi ]$ and $ A\langle \Phi \rangle$ are the same (\Cref{fun_comp_thm}). An example is also discussed in \Cref{Function_composition_section} which assert almost surjectivity of $\Phi$ is essential in order to prove the equality of two subalgebras. In fact, later in \Cref{Function_composition_section}, it is proved that an affirmative answer to the interpolation is equivalent to almost surjectivity of $\Phi$ under some additional hypothesis (\Cref{equivalent_interpolation_problem}). \\
\hf The interpolation problem has a few important and useful applications. In 1969, Ax  \cite{Ax} proved that any injective endomorphism of $X$ is an automorphism. Since then, this theorem had taken a central role in the affine algebraic geometry (for example, \cite{gromov,EGA-IV}). Several generalizations of the above result of Ax have been proposed. For example, Miyanishi proposed a conjectural generalization by weakening the injectivity hypothesis which has been partially confirmed by Kaliman \cite{Kaliman} and the author \cite{das}. We consider a generalization of the  theorem of Ax in a different direction. Our object of study is injective regular morphisms of affine varieties (not necessarily endomorphisms). As an application of the interpolation problem, we show the following:
\begin{app}[\Cref{set_th_criteria}]
Let $X,Y$ be affine varieties over $k$ such that $Y$ is factorial, and $\Phi:X \longrightarrow Y$ be a regular morphism. Then $\Phi$ is biregular if and only if it is injective and almost surjective.
\end{app}

\noindent This is in some sense a generalization of the result of Ax as set theoretic properties determine biregularity, a structural property. It is worthwhile to note that the above result immediately follows from the Zariski main theorem which states that a birational and quasi-finite morphism from an algebraic variety to a normal variety is an open immersion. Indeed $\Phi$, being injective, is quasi-finite and birational. The Zariski main theorem and Hartog's theorem about the extension of regular functions to the points of codimension at least 2 in normal varieties together imply that $\Phi$ is a surjective open immersion, and hence biregular. However, we prove \Cref{set_th_criteria} in a different approach (as an application of the interpolation problem). \\
\hf It is well known that a morphism $\Phi:X \longrightarrow Y$ of affine varieties is biregular if and only if the induced ring map $\Phi^*: A(Y) \longrightarrow A(X)$ is an isomorphism. If the base field is $\mathbb{C}$, the field of complex numbers, as another application of the interpolation problem, an analytic analog of the above result is developed. 
\begin{app}[\Cref{analytic_criterion_theorem}]
Let $X,Y$ be two complex affine varieties such that $Y$ is factorial. A regular map $\Phi : X \rightarrow Y$ is biregular if and only if the induced $\C$-algebra morphism $$\left(\Phi^{\text{an}} \right)^* : \Gamma (Y^{\text{an}}, \cH_Y) \longrightarrow \Gamma (X^{\text{an}},\cH_X)$$ is an isomorphism.
\end{app}
It is well known that a regular morphism of affine varieties is biregular if and only if the induced $\C$-algebra homomorphism between coordinate rings is an isomorphism. \Cref{analytic_criterion_theorem} can be thought of as an analytic analog that.

\section{Almost surjective morphisms of affine varieties}\label{Function_composition_section}
In this section, some properties of almost surjective morphisms of affine varieties are studied. Throughout this section, $k$ is always assumed to be an algebraically closed field of characteristic zero, unless specified otherwise. All the notations used here have usual meaning as defined in the previous section. Recall $\Phi:X \longrightarrow Y$ is always assumed to be a regular map,  $A \langle \Phi \rangle$ denotes the set of regular functions on $X$ that are determined by $\Phi$, and $A[\Phi ]$ denotes the image of $\Phi^*$ in $A(X)$. We start with the following property of almost surjective morphisms. The following lemma and its proof are generalizations of \cite[Lemma 3.2]{Aichinger}.
 \begin{lmm}\label{closure_lemma}
Let $X,Y$ be two affine varieties over $k$ such that $Y$ is factorial, and $\Phi : X \longrightarrow Y$ be an almost surjective morphism. Also let $g \in A \langle \Phi \rangle$. Consider the set $$D := \{ \left(\phi_1(x),\ldots ,\phi_n(x),g(x)\right) \in Y \times \mathbb{A}^1_k : x \in X \}.$$ Then $\dim \overline{D}=\dim Y$.
\end{lmm}
\begin{proof}
Let us assume that $\dim X=m_1$ and $\dim Y=n_1$. We also assume  $X\subseteq \mathbb{A}^m_k$ and $Y \subseteq \mathbb{A}^n_k$.\\
\hf Observe that the graph of the morphism $X\longrightarrow \mathbb{A}_k^{n+1}$, given by 
$$\textbf{x}=(x_1,\ldots ,x_m) \mapsto (\phi_1(\textbf{x}), \ldots , \phi_n(\textbf{x}),g(\textbf{x})),$$ 
is a closed algebraic subset of $\mathbb{A}_k^{m+n+1}$, and $D$ is the image of that algebraic subset under the natural projection $\pi_{n+1}: \mathbb{A}_k^{m+n+1} \longrightarrow \mathbb{A}_k^{n+1}$ onto the last $n+1$ components. By the Closure Theorem \cite[Theorem $1$, p.~258]{Cox_OShea}, there exists an algebraic set $W$ inside $Y\times \mathbb{A}^1_k$ such that $\overline{D}= D \cup W$ and $\dim W < \dim \overline{D}$. Observe that $\text{Im}(\Phi )=\pi(D)$, where $\pi : Y \times \mathbb{A}^1_k \rightarrow Y$ is the natural projection map. Note that $\Phi$, being almost surjective, is dominant. Therefore $\dim \overline{\pi(D)}=n_1$, and hence $\dim \overline{D} \geq n_1$. The latter inequality follows from the fact that $\dim V \geq \dim \overline{\pi(V)}$ for every algebraic set $V$ in $Y \times \mathbb{A}^1_k$.\\
\hf Assume on the contrary that $\dim \overline{D}=n_1+1$, that is, $ \overline{D}= Y \times \mathbb{A}^1_k$. Then $\dim \overline{\pi( W)} \leq \dim W \leq n_1$. We consider two different cases.\\
\hf First assume $\dim \overline{\pi( W)} = n_1$. In this case, $\overline{\pi( W)}=Y$, and hence the ideal of the algebraic set $\overline{\pi( W)}$ will be $I(Y)$, that is, $I(W)\cap k[X_1, \ldots ,X_{n}]= I(Y)$.
Since $\dim Y=n_1$, by Noether's normalization lemma there exist $n_1$ regular functions $h_1, \ldots,h_{n_1}$ on $Y$ which are algebraically independent over $k$. Therefore the subset $\{ h_1+I(W), \ldots ,h_{n_1}+I(W) \}$ of $A(Y)[X_{n+1}]/I(W)$, is algebraically independent over $k$ as well. Since $\dim W = n_1$, the set $\{ h_1+I(W), \ldots ,h_{n_1}+I(W), X_{n+1}+I(W) \}$ is algebraically dependent over $k$. Hence, there exists a polynomial $p$ of $n_1+1$ variables with coefficients from $k$ such that $p(h_1,\ldots ,h_{n_1},X_{n+1}) \in I(W)$. Let $q(x_1, \ldots ,x_n)\in A(Y)$ be the non-zero leading coefficient of $p(h_1,\ldots ,h_{n_1},X_{n+1}) \in A(Y)[X_{n+1}]$. Let $(y_1, \ldots ,y_n) \in Y$ be such that $q (y_1, \ldots ,y_n) \neq 0$, then there exist only finitely $z \in k$ such that $ (y_1, \ldots ,y_n,z) \in W$. Since $\overline{D}=Y \times k$ and $\overline{D}= D \cup W$, we conclude that $ (y_1, \ldots ,y_n,z)  \in D$ for all but (at most) finitely many $z$, which contradict the fact that $g$ is determined by $\Phi$.\\
\hf Next assume $\dim \overline{\pi( W)} < n_1$. In this case, $\overline{\pi( W)}$ is a proper subset of $Y$, and hence we can take $(y_1, \ldots ,y_n) \in Y \setminus \pi(W)$. Then for all $z \in k$, $(y_1, \ldots ,y_n,z) \in Y \times \mathbb{A}^1_k = \overline{D}$, but not in $W$. Therefore $(y_1, \ldots ,y_n,z) \in D$, which again contradict that $g$ is determined by $\Phi$.  
\end{proof}
Note that the $k$-subalgebra $A[\Phi ]$ is an integral domain; let us denote its field of fraction by $A\left( \Phi \, \right)$. The inclusion $A[\Phi ] \subseteq A(X) \cap A\left( \Phi \right)$ always holds. But the reverse inclusion does not hold in general. However, if we assume almost surjectivity of $\Phi$, the reverse inclusion also holds. This is a generalization of the implication $(i) \Rightarrow (ii)$ of \cite[Lemma 2.4]{Aichinger} indeed. In fact, later (in \Cref{equivalent_almost_surjective_lemma}) we show that this equality together with other assumptions is equivalent to the almost surjectivity of $\Phi$. 
\begin{lmm}\label{almost_surjective_lemma}
Let $X,Y$ be two affine varieties over $k$ such that $Y$ is factorial, and $\Phi : X \longrightarrow Y$ be an almost surjective morphism. Then $A(X) \cap A\left( \Phi \right)=A[\Phi ]$.
\end{lmm}
\begin{proof}
Clearly $A[\Phi ] \subseteq A(X) \cap A\left( \Phi \right)$. To prove the other way round, let $f \in A(X) \cap A\left( \Phi \right)$.
If $f=0$, we are done. Now assume $f$ is a non-zero element of $A(X)$ such that $f \in A\left( \Phi \right)$. Then there exist non-zero regular functions $g,h$ on $Y$ that are relatively prime in $A(Y)$ such that $$f= \dfrac{g \circ \Phi}{h \circ \Phi}.$$
Note that $\Phi$, being almost surjective, is dominant. Hence the induced map $\Phi^* \colon A(Y) \longrightarrow A(X)$ between the coordinate rings is injective, and $A(Y) \cong A[\Phi]$ via $\Phi^*$. Therefore $s \in A(Y)$ is unit in $A(Y)$ if and only if $s\circ \Phi$ is unit in $A[\Phi]$. To show $f \in A[\Phi ]$, it is enough to show that $h$ is a unit in $A(Y)$. Assume the contrary.
Let us consider the following sets
\begin{align*}
Z& =\{ y\in Y| h(y)=0 \}, \text{ and }\\
W& =\{y \in Y | g(y)=0\}.
\end{align*}
Since $h$ is non-unit, $Z \cap  \text{Im}(\Phi ) \neq \emptyset$. Otherwise, $Z \subseteq \overline{ Y \setminus  \text{Im}(\Phi )}$ and almost surjectivity of $\Phi$ contradict the fact that $\dim Z=\dim Y-1$. Moreover, dimension of $\overline{Z \cap  \text{Im}\left(\Phi  \right)}$ is $\dim Y-1$. This follows from almost surjectivity of $\Phi$, and the fact that $$Z \subseteq \overline{Z \cap  \text{Im}(\Phi )} \bigcup \overline{Z \cap \left( Y \setminus  \text{Im}(\Phi ) \right)}.$$ Since $f$ is a regular function on $X$, so $Z \cap  \text{Im}(\Phi ) \subseteq W \cap  \text{Im}(\Phi )$. Note that $g$ can't be a unit element of $A(Y)$; otherwise, $W \cap  \text{Im}(\Phi )= \emptyset$, which contradict that $Z \cap  \text{Im}(\Phi ) \neq \emptyset$. Thus we obtain $Z \cap W \cap  \text{Im}(\Phi ) = Z \cap  \text{Im}(\Phi )$, 
and hence $\dim  \overline{Z \cap W \cap  \text{Im}(\Phi )}= \dim \overline{ Z \cap  \text{Im}(\Phi )}= \dim Y-1$. But, since $g,h \in A(Y)$ are relatively prime, $\dim Z \cap W \leq \dim Y-2$, a contradiction. Hence, $h$ is a unit, and we are done.  
\end{proof}
We are now in a position to prove the main result in this context.
\begin{prp}\label{fun_comp_thm}
Let $X,Y$ be two affine varieties over $k$ such that $Y$ is factorial, and $\Phi : X \longrightarrow Y$ be an almost surjective morphism. Then $ A\langle \Phi \rangle= A[\Phi ]$.
\end{prp}
\begin{proof} Clearly $A[\Phi ] \subseteq A\langle \Phi \rangle$. The proof of the reverse inclusion is an adaption of the proof of \cite[Theorem 3.3]{Aichinger} to the more general situation. To prove it, let $g \in A\langle \Phi \rangle$. Clearly $g \in A(X)$ by definition.
 According to \Cref{almost_surjective_lemma}, it is enough to show that $g \in A\left( \Phi \right)$. Consider the set $$D := \{ \left(\phi_1(x),\ldots ,\phi_n(x),g(x)\right) \in Y \times \mathbb{A}^1_k \mid x \in X \}.$$ It follows from \Cref{closure_lemma} that $\overline{D}$ is an irreducible hypersurface in $Y \times k$. Let $\overline{D}$ is defined by a non-zero, non-unit irreducible polynomial $q \in A(Y)[X_{n+1}]$. Let $\deg_{X_{n+1}} (q)=d$. Observe that $d \neq 0$. Otherwise, $q \in A(Y)$, and thus $q \circ \Phi=0$. Following the fact that $\Phi$ is dominant, we conclude $q=0$, a contradiction. Therefore $d \geq 1$. We will now show that $d=1$. \\
\hf Denoting the function field of the variety $Y$ by $K=k(Y)$, we can assume $q$ as an element of $K[X_{n+1}]$. Let $r:= \text{Res}_{X_{n+1}}\left( q,\left( \partial/ \partial X_{n+1}\right)q \right)$, be the resultant of $q$ and $\left( \partial/ \partial X_{n+1}\right)q$, the derivative of $q$ with respect to $X_{n+1}$. Following the notations of the proof of \Cref{closure_lemma}, we obtain there exists an algebraic subset $W$ of $Y \times \mathbb{A}^1_k$ such that $\overline{D}= D \cup W$ and $\dim W < \dim \overline{D}=n_1$. \\
\hf First assume $r=0$, then both $q$ and $\left( \partial/ \partial X_{n+1}\right)q$ have a non-constant common factor in $K[X_{n+1}]$\footnote{For $f,g \in K[X]$, if $r= \text{Res}_X (f,g)$, then $r=0$ if and only if both $f$ and $g$ have a non-constant common factor in $K[X]$ (cf.~\cite{Cox_OShea}).}. Since $q$ is irreducible in $ A(Y)[X_{n+1}]$, by the Gauss Lemma, $q$ is also irreducible in $ K[X_{n+1}]$ which contradict the fact that $q$ has a factor in $K[X_{n+1}]$. \\
\hf Therefore $r \neq 0$. If $q_d\in A(Y)$ denotes the leading coefficient of $q$, there exists an element $\textbf{y} \in Y$ such that $q_d(\textbf{y} ) \neq 0, r(\textbf{y} ) \neq 0$ and $\textbf{y}  \notin \pi(W)$. Such an $\textbf{y}$ exists because $Z(q_d) \cup Z(r)\cup \overline{\pi(W)} \neq Y$, where $Z(h)$ denotes the zero set of a function $h$.  If we denote the polynomial $q(\textbf{y} ,z)$ in $k[z]$ by $\tilde{q}(z)$, then $r(\textbf{y})= \text{Res}_z(\tilde{q}(z),\tilde{q}^{\prime}(z)) $. Since $r(\textbf{y} ) \neq 0$, the polynomial $\tilde{q}(z)$ has $d$ distinct roots in $k$. If $b \in k$ such that $q(\textbf{y} ,b)=0$, then $(\textbf{y} ,b) \in \overline{D}$, but not in $W$, that is, $(\textbf{y} ,b)\in D$. We deduce $(\textbf{y} ,b)\in D$ if and only if $q(\textbf{y} ,b)=0$. Since $g$ is determined by $\Phi$, there exists at most one $b \in k$ such that $(\textbf{y} ,b)\in D$. Thus we obtain $d \leq 1$. Therefore $d=1$, and hence $g \in A(\Phi )$.
\end{proof}
\begin{rem}
In the proof, we have used the fact that \textit{char} $k=0$. This is because the degree of $\left( \partial/ \partial X_{n+1}\right)q$ with respect to the variable $X_{n+1}$ is $d-1$, and $\left( \partial/ \partial X_{n+1}\right)q \neq 0$. But if \textit{char} $k=p>0$, then it could happen that $\left( \partial/ \partial X_{n+1}\right)q = 0$, or $\deg_{X_{n+1}} \left( \partial/ \partial X_{n+1}\right)q = 0$, even when $d>1$. One need to take care of the above possibilities when the base field is of positive characteristic. In fact, if $\left( \partial/ \partial X_{n+1}\right)q = 0$, then $q \in A(Y)[X_{n+1}^{p^n}]$ for some $n \in \mathbb{N}$, and it can be shown that $g^{p^n} \in A[\Phi]$. So the equality of $A \langle \Phi \rangle$ and $A[\Phi ]$ can't be expected if  \textit{char} $k>0$. An example is given below in this context. Since we are interested about the equality of the two sub-algebras, the positive characteristic case is not included here.
\end{rem}
The following example, discussed in \cite[p. 308]{Aichinger}, illustrates the difficulty in positive characteristics.
\begin{eg}
Let $k$ be the algebraic closure of a simple field $\mathbb{F}_p$ of characteristic $p$. Assume $X=Y=\mathbb{A}^1_k$, and $\Phi: X \longrightarrow Y$, given by $$\Phi(x)=x^p.$$ Clearly $\Phi$ is bijective, but not an automorphism. Thus $A[\Phi ] \neq A(X)$. Observe that $A \langle \Phi \rangle = A(X)$. 
\end{eg}
The interpolation problem is now an immediate consequence of \Cref{fun_comp_thm}.
\begin{thm}\label{fun_comp_corollary}
Let $X,Y$ be two affine varieties over $k$ such that $Y$ is factorial, and $\Phi : X \longrightarrow Y$ be an almost surjective morphism. Furthermore, assume $h : Y \rightarrow k $ is a function such that $h \circ \Phi$ is a regular function on $X$. Then there exists a regular function $p:Y \rightarrow k$ such that $p(y)=h(y)$ for all $y \in \Phi (X)$.
\end{thm}
\begin{proof}
Note that by definition $h \circ \Phi \in A \langle \Phi \rangle$. According to \Cref{fun_comp_thm}, $h \circ \Phi \in A[\Phi]$, that is, $h \circ \Phi= p \circ \Phi$ for some $p \in A(Y)$. The fact $p(y)=h(y)$ for all $y \in \text{Im}\left( \Phi \right)$ is clear from the discussion.
\end{proof}
\begin{rem}
\Cref{fun_comp_corollary} can be though of as a generalization of the following fact from one variable complex calculus: \\
Let $g$ be a non-constant polynomial over $\mathbb{C}$, and $f: \mathbb{C} \longrightarrow \mathbb{C}$ is a function such that $f\circ g$ is a polynomial. Then $f$ is a polynomial.
\end{rem} 
We now consider an example which is discussed in \cite[p. 304]{Aichinger}. It asserts that almost surjectivity of $\Phi$ is necessary in \Cref{fun_comp_thm}.
In fact, shortly we show that with some other mild hypothesis, almost surjectivity is necessary and sufficient for an affirmative answer to the interpolation problem (\Cref{equivalent_interpolation_problem}). 
\begin{eg}
Let $X=Y= \mathbb{A}^2_{k}$ and $\Phi: X \longrightarrow Y$, be defined by $\Phi(x,y)= (x,xy).$ Also consider the function $h: Y \rightarrow k$, given by
\begin{align*}
h(x,y)= \begin{cases}
y^2/x & \text{ if } x\neq 0,\\
0 & \text{otherwise.}
\end{cases}
\end{align*}
Clearly $h \circ \Phi \in A \langle \Phi \rangle$. Also note that $h \circ \Phi$ can be written as $(y^2 \circ \Phi )/ (x \circ \Phi)$, that is, $h \circ \Phi \in A(\Phi)$ as well. We now show that $h \circ \Phi$ does not belong to $A[ \Phi ]$. Note that $\Phi$ is dominant, that is, $A(Y) \cong A[\Phi]$ via $\Phi^*$. Now, $h \circ \Phi \in A[\Phi]$ if and only if $x \circ \Phi$ is unit in $A[\Phi]$. It is equivalent to saying that the regular function $x$ on $Y$ is invertible, which is not true. Thus $h \circ \Phi \notin A[\Phi]$. Observe that $\Phi$ is not almost surjective; all the other hypotheses of \Cref{fun_comp_thm} are satisfied.
\end{eg}
\subsection{Criterion for almost surjectivity}
It is seen that almost surjectivity of a regular map plays a crucial role in \Cref{fun_comp_thm} and \Cref{fun_comp_corollary}. We now discuss various equivalent criteria of almost surjectivity of a morphism which will be useful for our purpose. We start by proving an immediate consequence about constructible subset of an affine variety. This is an immediate extension of \cite[Proposition 2.3]{Aichinger}. The proof is included for the sake of completion.

\begin{lmm}\label{constructible_lemma}
Let $X$ be an affine variety of dimension $n$ over $k$, and $S$ be a constructible subset of $X$ such that $\dim \overline{S} \geq n-1$. Then there exist algebraic sets $Y,Z$ such that $Y $ is irreducible, $\dim Y= n-1$, $\dim Z \leq n-2$, and $Y \setminus Z \subseteq S$.
\end{lmm}
\begin{proof}
Since $S$ is constructible, by definition, there exist algebraic subvarieties $Y_1,\ldots ,Y_p$ of $X$, and closed algebraic subsets $Z_1, \ldots ,Z_p$ of $X$ with $Z_i \subsetneq Y_i$ for all $i$, and $S= \bigcup \limits_{i=1}^p \left( Y_i \setminus Z_i\right)$. Without loss of generality let us assume $\dim Y_1$ is the maximum among all $\dim Y_i$ for all $i$. Since $\overline{S} \subseteq \bigcup \limits_{i=1}^p Y_i$ and $\dim \overline{S} \geq n-1$, we conclude that $\dim Y_1 \geq n-1$.\\
\hf If $\dim Y_1=n-1$, then $Y= Y_1$, and $Z=Z_1$ are the required sets.\\
\hf If $\dim Y_1=n$, then $Y_1=X$, and $X \setminus Z_1 \subseteq S$. Let $W$ be a codimension $1$ subvariety of $X$ such that $W \nsubseteq Z_1$. Then clearly $W \cap \left( X \setminus Z_1\right)= W \setminus \left( W \cap Z_1 \right)$, and hence $W \cap \left( X \setminus Z_1\right)  \subseteq S$. Moreover, $W \cap Z_1 \neq W$, thus $Y=W$ and $Z= W \cap Z_1$ are the required sets.
\end{proof}
The following result gives equivalent criteria of almost sujectivity of regular maps from any affine variety to an affine factorial variety. It is a generalization of Lemma 2.4 and Theorem 3.3 from \cite{Aichinger}.
\begin{prp}\label{equivalent_almost_surjective_lemma}
Let $X,Y$ be two algebraic varieties over $k$ such that $Y$ is factorial, and $\Phi:X \longrightarrow Y$ be a regular morphism. Then the following statements are equivalent:
\begin{enumerate}[label=(\roman*)]
\item \label{i} $\Phi$ is almost surjective.
\item \label{ii} $\Phi$ is dominant, and $A(X) \cap A\left( \Phi \right)=A[\Phi ]$.
\item \label{iii} For $f,g \in A(Y)$ with $f \circ \Phi \mid g \circ \Phi$ in $A(X)$, we have $f \mid g$ in $A(Y)$.
\item \label{iv} $\Phi$ is dominant, and $A \langle \Phi \rangle=A[\Phi ]$.
\end{enumerate}
\end{prp} 
\begin{proof}
\ref{i} $\Rightarrow$ \ref{ii} The morphism $\Phi$, being almost surjective, is dominant. Other conclusion follows from \Cref{almost_surjective_lemma}.\\
\ref{ii} $\Rightarrow$ \ref{iii} Since, $\Phi$ is dominant, $\Phi^*$ is injective. By the hypothesis, there exists $h \in A(X) $ such that $g \circ \Phi= h \cdot \left( f \circ \Phi \right)$. If $f \circ \Phi=0$, then $g \circ \Phi=0$. It follows from injectivity of $\Phi^* $ that both $f,g$ are zero in $A(Y)$, and in this case $f \mid g$. Now if $f \circ \Phi \neq 0$, then $h \in A(X) \cap A\left( \Phi \right)$, that is, $h \in A[\Phi]$ by the hypothesis. Thus there exists $p \in A(Y)$ such that $p \circ \Phi =h$. Hence we get $\left( g- p f\right) \circ \Phi =0$. Injectivity of $\Phi^* $ yields $g- p f=0$. Thus $f \mid g$ in $A(Y)$.\\
\ref{iii} $\Rightarrow$ \ref{i} We prove this by the method of contradiction. Assume $\Phi$ is not almost surjective. Then $S := Y \setminus \text{Im}(\Phi )$ is a constructible subset of $Y$ such that $\dim \overline{S} \geq \dim Y -1$. Then by \Cref{constructible_lemma}, there exist algebraic subsets $Z,W$ of $Y$ such that $Z $ is irreducible, $\dim Z= \dim Y -1$, $\dim W \leq \dim Y -2$, and $Z \setminus W \subseteq S$. Since $A(Y)$ is a UFD, there exists an irreducible regular function $p$ on $Y$ such that $Z=V(p)$. Observe that due to dimensional reason $Z \nsubseteq W$, that is, $I(W) \nsubseteq I(Z)$. Thus there exists $q \in A(Y)$ such that $q \in I(W)$ with $q \notin I(Z)$. Note that $Z \subseteq W \cup S$, and hence $Z \cap \text{Im}(\Phi ) \subseteq W $. Therefore if for $x \in X$, $p\left( \Phi (x) \right)=0$, then $q\left( \Phi (x) \right)=0$. By Hilbert Nullstellensatz, we deduce that $q \circ \Phi \in \sqrt{p \circ \Phi}$. There exists $n\in \mathbb{N}$ and $r \in A(X)$ such that $q \circ \Phi = r \cdot \left( p \circ \Phi \right)^n$, and hence $p\circ \Phi \mid q \circ \Phi$. Using \ref{iii}, we conclude that $p \mid q$; thus $q \in I(Z)$, a contradiction. \\
Thus we have proved \ref{i} $\Leftrightarrow$ \ref{ii} $\Leftrightarrow$ \ref{iii}. Now we show that \ref{i} $\Leftrightarrow$ \ref{iv}. Again, \ref{i} $\Rightarrow$ \ref{iv} follows from \Cref{fun_comp_thm}. \\ 
\hf To show \ref{iv} $\Rightarrow$ \ref{i}, assume the contrary, i.e., $\Phi$ is not almost surjective. Then by \Cref{equivalent_almost_surjective_lemma}, there exist regular functions $f$ and $g$ on $Y$ such that $f \circ \Phi \mid g \circ \Phi$, but $f \nmid g$. Moreover, we can assume $\gcd (f,g)=1$. There exists a regular function $h$ on $X$ such that $g \circ \Phi = h \cdot \left( f \circ \Phi \right)$. Consider the regular function $\tilde{g}:= h \cdot \left( g \circ \Phi \right)$ on $X$. We show that $\tilde{g} \in A \langle \Phi \rangle$, but it does not belong to $A [ \Phi ]$.\\
\hf To show $\tilde{g} \in A \langle \Phi \rangle$, we assume $a,b \in X$ such that $\Phi (a)= \Phi (b)$. If $\left( g \circ \Phi \right) (a)=0$, then $\tilde{g}(a)=\tilde{g}(b)$. If $\left( g \circ \Phi \right) (a) \neq 0$, then $$\tilde{g}(a)= \dfrac{ \left( g \circ \Phi \right)^2(a) }{ \left( f \circ \Phi \right)(a)}=\frac{\left( g \circ \Phi \right)^2(b)}{ \left( f \circ \Phi \right)(b)}= \tilde{g}(b).$$ 
Thus we get for $a,b \in X$, if $\Phi (a)= \Phi (b)$, $\tilde{g}(a)=\tilde{g}(b)$. It is now easy to construct a function $p$ on $Y$ such that $p \circ \Phi= \tilde{g}$. Therefore, $\tilde{g} \in A \langle \Phi \rangle$. \\
\hf Now our goal is to show $\tilde{g} \notin A[ \Phi ]$. 
Assume $\tilde{g} \in A[\Phi ]$, that is, there exists a regular function $\tilde{h}$ such that $\tilde{g}= \tilde{h} \circ \Phi$. Thus we derive $(  \tilde{h} \circ \Phi )\left(  f \circ \Phi \right)=\left(  g \circ \Phi \right)^2 $. Since $\Phi$ is assumed to be dominant, $\tilde{h} f= g^2$, and hence $f \mid g^2$ which contradict that $\gcd (f,g)=1$. Thus $\tilde{g} \notin A[\Phi ]$, a contradiction. Hence $\Phi$ is almost surjective.
\end{proof}
The following theorem is an immediate consequence of the above result:
\begin{thm}\label{equivalent_interpolation_problem}
Let $X,Y$ be two algebraic varieties over $k$ such that $Y$ is factorial, and $\Phi:X \longrightarrow Y$ be a dominant regular morphism. The interpolation problem has an affirmative answer if and only if $\Phi$ is almost surjective.
\end{thm}
\section{Applications of the interpolation problem}
In this section, some important applications of \Cref{fun_comp_thm} and \Cref{fun_comp_corollary} are discussed. Throughout this section, we assume $\Phi:X \longrightarrow Y $ is a regular morphism and $Y$ is factorial.
\subsection{A generalization of a theorem of Ax}
As mentioned in \Cref{intro}, we consider a generalization of the theorem of Ax in a different direction. Our object of study here is injective regular map of affine varieties such that the target variety is factorial. We always assume that $Y$ is factorial and $\Phi:X \longrightarrow Y$ is an injective regular map, unless specified otherwise. The following theorem is another important consequence of \Cref{fun_comp_thm}. 
\begin{thm}\label{set_th_criteria}
Let $X,Y$ be affine varieties over $k$ such that $Y$ is factorial, and $\Phi:X \longrightarrow Y$ be a regular morphism. Then $\Phi$ is biregular if and only if it is injective and almost surjective.
\end{thm}
\begin{proof}
For a set $X$, let us denote the ring of $k$-valued function on $X$ by $\mathcal{F}(X, k)$. Since $\Phi$ is injective, the induced ring map $$\widetilde{\Phi}: \mathcal{F}(Y, k) \longrightarrow \mathcal{F}(X, k),$$ given by the composition with $\Phi$, is surjective. We need to show that $\Phi^*: A(Y) \longrightarrow A(X)$ is an isomorphism. Clearly $\Phi$ is dominant, and hence $\Phi^*$ is injective. To show surjectivity of $\Phi^*$, let $f \in A(X)$. Considering $f$ as an element of  $\mathcal{F}(X, k)$, there exists $g \in \mathcal{F}(Y, k)$ such that $\widetilde{\Phi} (g)=f$, that is, $g \circ \Phi = f$. Thus $f \in A\langle \Phi \rangle$, and by \Cref{fun_comp_thm}, $f \in A[\Phi]$. Therefore, $\Phi^*$ is surjective, and hence $\Phi$ is biregular. The other way is immediate.
\end{proof}
The following results are immediate consequences of the above result.
\begin{cor}\label{main_corollary}
Let $X,Y$ be affine varieties over $k$ such that $Y$ is factorial, and $\Phi:X \longrightarrow Y$ be a regular morphism. If $\Phi$ is bijective, then it is biregular.
\end{cor}

\subsection{Analytic criterion of biregularity}
We now move towards another important application of the interpolation problem \textemdash \, analytic criterion of biregularity of a regular morphism of complex affine varieties. The base field is assumed to be $\mathbb{C}$ in this case. For more details about the analytic properties of complex algebraic varieties, morphisms, and their connections with the algebraic properties, we refer the reader to \cite{GAGA}. A regular morphism $\Phi:X \longrightarrow Y$ induces a holomorphic map $\Phi^{\text{an}}: X^{\text{an}}  \longrightarrow Y^{\text{an}}$. It is well known that $\Phi$ is biregular if and only if $\Phi^{\text{an}}$ is biholomorphic (cf. \cite[Proposition 9]{GAGA}). Therefore the induced $\C$-algebra homomorphism  $$\left(\Phi^{\text{an}} \right)^* : \Gamma (Y^{\text{an}}, \cH_Y) \longrightarrow \Gamma (X^{\text{an}},\cH_X)$$ is an isomorphism, in particular, if $\Phi$ is biregular. However, the converse is not true in general, that is, isomorphism of $\left(\Phi^{\text{an}} \right)^*$ does not always yield biregularity of $\Phi$. However, under some suitable hypothesis, the converse is also true.  
\begin{thm}\label{analytic_criterion_theorem}
Let $X,Y$ be two complex affine varieties such that $Y$ is factorial. Assume $\Phi : X \rightarrow Y$ is a regular morphism. Then the following statements are equivalent:
\begin{enumerate}[label=(\roman*)]
\item \label{1_isom} The morphism $\Phi$ is biregular.
\item \label{2_isom} The induced $\C$-algebra morphism $$\left(\Phi^{\text{an}} \right)^* : \Gamma (Y^{\text{an}}, \cH_Y) \longrightarrow \Gamma (X^{\text{an}},\cH_X)$$ is an isomorphism. 
\end{enumerate}
\end{thm}
\begin{proof}
\ref{1_isom} $\Rightarrow$ \ref{2_isom} is easy. To show \ref{2_isom} $\Rightarrow$ \ref{1_isom}, consider the following commutative diagram of rings.
\begin{equation}\label{mixed_sheaf_diag_another}
\begin{gathered}
\xymatrix{
A(Y) \ar[rr]^-{\Phi^*} \ar[d] && A(X) \ar[d]\\
\Gamma(Y^{\text{an}}, \cH_Y) \ar[rr]^-{\left(\Phi^{\text{an}} \right)^*}_-{\simeq}& & \Gamma(X^{\text{an}},\cH_X)
}
\end{gathered}
\end{equation} 
Note that all the vertical arrows are injective. The commutativity of the diagram implies that $\Phi^*$ is injective. \\
\hf We first show that $\Phi$ is almost surjective. Let $f,g \in A(Y)$ with $f \circ \Phi \mid g \circ \Phi$ in $A(X)$. Then there exists $h \in A(X)$ such that $g \circ \Phi= h \cdot \left( f \circ \Phi \right)$. If $ f \circ \Phi=0$, then $ g \circ \Phi=0$, and hence by injectivity of $\Phi^*$, we conclude that both $f,g$ are zero. Thus $f \mid g$ in this case. Now assume $ f \circ \Phi \neq 0$. Since $\left(\Phi^{\text{an}} \right)^*$ is an isomorphism, there exists $p \in \Gamma(Y^{\text{an}}, \cH_Y)$ such that $p \circ \Phi = h$. We deduce that $g=pf$ in $\Gamma(Y^{\text{an}}, \cH_Y)$. Therefore $f \mid g$ in $\cH_{Y,y}$ for all $y \in Y$. Hence $f \mid g $ in $\widehat{\cH}_{Y,y}$, the completion of the local ring $\cH_{Y,y}$ with respect to its unique maximal ideal. If we denote the completion of the local ring $\cO_{Y,y}$ with respect to its unique maximal ideal by $\widehat{\cO}_{Y,y}$, then $\widehat{\cH}_{Y,y}$ and $\widehat{\cO}_{Y,y}$ are naturally isomorphic (cf. \cite{GAGA}). We conclude that $f \mid g $ in $\widehat{\cO}_{Y,y}$. Again, by \cite[Lemma 1.29, Chapter 1]{Mumford}, $f \mid g $ in $\cO_{Y,y}$. Considering $g/f$ as an element of $k(Y)$, the function field of $Y$, we conclude that $g/f \in \cO_{Y,y}$ for all $y \in Y$. Hence $g/f \in A(Y)$, that is, $f \mid g$ in $A(Y)$. By \Cref{equivalent_almost_surjective_lemma}, $\Phi$ is almost surjective.\\
\hf To prove $\Phi$ is biregular, it is enough to show that $\Phi^*$ is bijective. As discussed earlier, $\Phi^*$ is injective. To show $\Phi^*$ is surjective, let $r \in A(X)$. As $\left(\Phi^{\text{an}} \right)^*$ is bijective, there exist $s \in \Gamma(Y^{\text{an}}, \cH_Y) $ such that $s \circ \Phi =r$. By \Cref{fun_comp_thm}, there exists $t \in A(Y)$ such that $t \circ \Phi=s \circ \Phi$, that is, $t \circ \Phi=r$. Surjectivity of $\Phi^*$ follows.
\end{proof}

\section*{Acknowledgment} \noindent The author would like to thank Chitrabhanu Chaudhuri, Ritwik Mukherjee, and Zbigniew Jelonek for several fruitful discussions and comments. Conversations with Erhard Aichinger over email were helpful. Special thanks are due to him for pointing out reference \cite{Aichinger}, which the author had looked at for a different purpose before beginning the project. The author is grateful to Chitrabhanu Chaudhuri for providing an example to sort out the mistake that the author had made earlier. This work is supported by the INSPIRE faculty fellowship (Registration No.: IFA21-MA 161) funded by the DST, Govt. of India. The author is also thankful to the anonymous referee for suggesting several improvements to the earlier manuscript.

\bibliographystyle{abbrv}
\bibliography{ref}

\begin{thebibliography}{10}

\bibitem{Aichinger}
E.~Aichinger.
\newblock On function compositions that are polynomials.
\newblock {\em J. Commut. Algebra}, 7(3):303--315, 2015.

\bibitem{Ax}
J.~Ax.
\newblock Injective endomorphisms of varieties and schemes.
\newblock {\em Pacific J. Math.}, 31:1--7, 1969.

\bibitem{Cox_OShea}
D.~Cox, J.~Little, and D.~O'Shea.
\newblock {\em Ideals, varieties, and algorithms}.
\newblock Undergraduate Texts in Mathematics. Springer, New York, third
  edition, 2007.
\newblock An introduction to computational algebraic geometry and commutative
  algebra.

\bibitem{das}
N.~Das.
\newblock {On Endomorphisms of Algebraic Varieties}.
\newblock {\em International Mathematics Research Notices},
  2022(22):17534--17545, 08 2021.

\bibitem{gromov}
M.~Gromov.
\newblock Endomorphisms of symbolic algebraic varieties.
\newblock {\em J. Eur. Math. Soc. (JEMS)}, 1(2):109--197, 1999.

\bibitem{EGA-IV}
A.~Grothendieck.
\newblock \'{E}l\'{e}ments de g\'{e}om\'{e}trie alg\'{e}brique. {IV}. \'{E}tude
  locale des sch\'{e}mas et des morphismes de sch\'{e}mas {IV}.
\newblock {\em Inst. Hautes \'{E}tudes Sci. Publ. Math.}, (32):361, 1967.

\bibitem{Kaliman}
S.~Kaliman.
\newblock On a theorem of {A}x.
\newblock {\em Proc. Amer. Math. Soc.}, 133(4):975--977, 2005.

\bibitem{Mumford}
D.~Mumford.
\newblock {\em Algebraic geometry. {I}}.
\newblock Springer-Verlag, Berlin-New York, 1976.
\newblock Complex projective varieties, Grundlehren der Mathematischen
  Wissenschaften, No. 221.

\bibitem{Rudin}
W.~Rudin.
\newblock {\em Real and complex analysis}.
\newblock McGraw-Hill Book Co., New York, third edition, 1987.

\bibitem{GAGA}
J.-P. Serre.
\newblock G\'{e}om\'{e}trie alg\'{e}brique et g\'{e}om\'{e}trie analytique.
\newblock {\em Ann. Inst. Fourier (Grenoble)}, 6:1--42, 1955/56.

\end{thebibliography}
\end{document}